\documentclass[12 pt]{article}
\usepackage{amsmath, amsthm, amssymb,enumerate,verbatim,authblk}
\usepackage{tikz}
\usetikzlibrary{shapes}
\usetikzlibrary{calc}
\usepackage{fullpage}
\usepackage{hyperref}

\title{Saturation in the Hypercube and Bootstrap Percolation} 
\author{Natasha Morrison}
\author{Jonathan A. Noel}
\author{Alex Scott}

\affil{\normalsize{Mathematical Institute, University of Oxford, Andrew Wiles Building, Radcliffe Observatory Quarter, Woodstock Road, Oxford OX2 6GG, UK.}}
\affil{\texttt{\{morrison,noel,scott\}@maths.ox.ac.uk}}

\newtheorem{thm}[equation]{Theorem}
\newtheorem{lem}[equation]{Lemma}

\newtheorem{claim}[equation]{Claim}
\newtheorem{subclaim}[equation]{Subclaim}

\newtheorem{prob}[equation]{Problem}
\newtheorem{ques}[equation]{Question}

\theoremstyle{definition}

\newtheorem{obs}[equation]{Observation}

\theoremstyle{remark}
\newtheorem{rem}[equation]{Remark}
\newtheorem{step}{Step}

\newtheoremstyle{case}{}{}{\normalfont}{}{\itshape}{\normalfont:}{ }{}

\theoremstyle{case}
\newtheorem{case}{Case}
\newtheorem{case2}{Case}

\newcommand{\sat}{\operatorname{sat}}
\newcommand{\ex}{\operatorname{ex}}
\newcommand{\wsat}{\operatorname{wsat}}
\newcommand{\ssat}{\operatorname{ssat}}
\newcommand{\vspan}{\operatorname{span}}
\newcommand{\vdim}{\operatorname{dim}}

\numberwithin{equation}{section}

\begin{document}

\maketitle

\begin{abstract}
Let $Q_d$ denote the hypercube of dimension $d$. Given $d\geq m$, a 
spanning subgraph $G$ of $Q_d$ is said to be \emph{$(Q_d,Q_m)$-saturated} 
if it does not contain $Q_m$ as a subgraph but adding any edge of 
$E(Q_d)\setminus E(G)$ creates a copy of $Q_m$ in $G$. Answering a 
question of Johnson and Pinto~\cite{hypercube}, we show that for every 
fixed $m\geq2$ the minimum number of edges in a $(Q_d,Q_m)$-saturated 
graph is $\Theta(2^d)$.

We also study weak saturation, which is a form of bootstrap percolation. 
A spanning subgraph of $Q_d$ is said to be \emph{weakly $(Q_d,Q_m)$-saturated} 
if the edges of $E(Q_d)\setminus E(G)$ can 
be added to $G$ one at a time so that each added edge creates a new 
copy of $Q_m$.  Answering another question of Johnson and 
Pinto~\cite{hypercube}, we determine the minimum number of edges in a 
weakly $(Q_d,Q_m)$-saturated graph for all $d\geq m\geq1$. More generally, we determine the minimum number of edges in a subgraph of the $d$-dimensional grid $P_k^d$ which is weakly saturated with respect to `axis aligned' 
copies of a smaller grid $P_r^m$.  We also study weak saturation of cycles in the grid.
\end{abstract}

\section{Introduction}

Given graphs $F$ and $H$, a spanning subgraph $G$ of $F$ is said to be \emph{$(F,H)$-saturated} if it does not contain $H$ as a subgraph, but for every edge $e\in E(F)\setminus E(G)$, the graph $G+e$ contains a copy of $H$. In this language, the classical Tur\'{a}n problem asks for the maximum size of an $(F,H)$-saturated graph; this number is known as the \emph{extremal number}, denoted $\ex(F,H)$. Another well-studied problem, introduced independently by Zykov~\cite{Zykov} and Erd\H{o}s, Hajnal and Moon~\cite{EHM}, is to determine the minimum number of edges in an $(F,H)$-saturated graph. This is known as the \emph{saturation number}, denoted $\sat(F,H)$. A spanning subgraph $G$ of $F$ is said to be \emph{weakly $(F,H)$-saturated} if the edges of $E(F)\setminus E(G)$ can be added to $G$, one edge at a time, in such a way that every added edge creates a new copy of $H$. This notion was introduced by Bollob\'{a}s~\cite{wsat}. The minimum number of edges in a weakly $(F,H)$-saturated graph is known as the \emph{weak saturation number} and denoted $\wsat(F,H)$. Note that every $(F,H)$-saturated graph is also weakly $(F,H)$-saturated and so
\begin{equation}\label{wsatlower}\wsat(F,H)\leq \sat(F,H).\end{equation}
For additional background on minimum saturation in graphs, see~\cite{satsurvey}.

Given $k\geq2$ and $d\geq1$, the \emph{$k$-grid of dimension $d$} is the graph $P_k^d$ with vertex set $\{0,1,\dots,k-1\}^d$ where two vertices are adjacent 
if they differ by one in some coordinate and are equal in the other $d-1$ coordinates. 
In particular, $P_2^d$ is known the \emph{hypercube of dimension $d$} and denoted $Q_d$. In this paper, we are interested in saturation and weak saturation problems in hypercubes and grids. 

With regards to $\sat\left(Q_d,Q_m\right)$, the special case $m=2$ was studied by Choi and Guan~\cite{Choi} who constructed a $(Q_d,Q_2)$-saturated graph with at most $\left(\frac{1}{4}+o(1)\right)|E(Q_d)|$ edges. Santolupo (see~\cite{satsurvey}) conjectured that this construction is best possible. In their recent paper~\cite{hypercube}, Johnson and Pinto disproved this conjecture (in a strong sense) by showing that, for every fixed $m$, there exists a $(Q_d,Q_m)$-saturated graph with $o(|E(Q_d)|)$ edges. More precisely, they proved the following.

\begin{thm}[Johnson and Pinto~\cite{hypercube}]
\label{pinto}
For every fixed $m\geq2$, there exists $0<\varepsilon_m<1$ such that $\sat(Q_d,Q_m) = O\left(d^{(1-\varepsilon_m)}2^d\right) = O(|E(Q_d)|/d^{\varepsilon_m})$.
\end{thm}

In the case $m=2$, Johnson and Pinto~\cite{hypercube} obtained a stronger bound; namely, 
\begin{equation}\label{Q2bound}\sat(Q_d,Q_2) < 10\cdot 2^d.\end{equation}
That is, for every $d$, there exists a $(Q_d,Q_2)$-saturated graph with bounded average degree. Motivated by this result, they asked the following: for which fixed values of $m$ is
\[\sat(Q_d,Q_m) = O\left(2^d\right)?\]
We show that this is the case for every $m\geq2$.

\begin{thm}
\label{main}
For every fixed $m\geq2$, $\sat\left(Q_d,Q_m\right)\leq (1+o(1))72m^22^d$.
\end{thm}

We remark that, previously, the best known lower bound on $\sat(Q_d,Q_m)$ was $(m+1-o(1))2^{d-1}$, due to Johnson and Pinto~\cite{hypercube}. In particular, for fixed $m$, Theorem~\ref{main} is tight up to a (constant) factor of $O(m)$.

In the case of weak saturation, Johnson and Pinto~\cite{hypercube} proved
\begin{equation}\label{2^d-1}\wsat(Q_d,Q_2)=2^d-1\end{equation}
for all $d\geq2$ by exhibiting a spanning tree of $Q_d$ which is weakly $(Q_d,Q_2)$-saturated.  They  asked about the value of $\wsat(Q_d,Q_m)$ for general $d\geq m\geq1$. We answer this question. 

\begin{thm}
\label{wsatThmQ}
For $d\geq m\geq 1$, 
\[\wsat\left(Q_d,Q_m\right) = (m-1)2^d - \sum_{j=0}^{m-2}(m-1-j)\binom{d}{j}.\]
\end{thm}

By combining Theorem~\ref{wsatThmQ} with  (\ref{wsatlower}), we can improve the lower bound for $\sat(Q_d,Q_m)$ mentioned above to $(m-1-o(1))2^d$.  

Theorem~\ref{wsatThmQ} is a special case of a more general result. Given $k\geq r\geq2$ and $d\geq m\geq1$, say that a copy of $P_r^m$ in $P_k^d$ is \emph{axis aligned} if it is induced by a set of vertices of the form $I_1\times \dots \times I_d$ where exactly $m$ of the sets $I_i$ are intervals of length $r$ in $\{0,\dots,k-1\}$ and the rest are singletons. Let $\wsat^*\left(P_k^d,P_r^m\right)$ be the minimum number of edges in a spanning subgraph $G$ of $P_k^d$ such that the edges of $E\left(P_k^d\right)\setminus E(G)$ can be added to $G$, one edge at a time, such that every added edge creates an axis aligned copy of $P_r^m$. We prove the following.

\begin{thm}
\label{wsatThm}
For $k\geq r\geq 2$ and $d\geq m\geq 1$, 
\[\wsat^*\left(P_k^d,P_r^m\right) = \sum_{j=0}^d\sum_{i=0}^{d-j}(m-1+i)\binom{d}{j}\binom{d-j}{i}(k-r+1)^j(r-2)^i\]
\[-\sum_{j=0}^{m-2}\sum_{i=0}^{d-j}(m-1-j)\binom{d}{j}\binom{d-j}{i}(k-r+1)^j(r-2)^i,\]
where, by convention, $0^0=1$.
\end{thm}


We remark that if $r=2$ or $d=m$, then every copy of $P_r^m$ in $P_k^d$ is axis aligned. Thus, in these cases, $\wsat^*\left(P_k^d,P_m^r\right) = \wsat\left(P_k^d,P_m^r\right)$. In particular, Theorem~\ref{wsatThmQ} is implied by the case $k=r=2$ of Theorem~\ref{wsatThm}. 

We also consider an extension of (\ref{2^d-1}) to general even cycles in the grid, proving the following. 

\begin{thm}
\label{cycleThm}
For $k\geq2$ and $d\geq \ell\geq2$, $\wsat\left(P_k^d,C_{2\ell}\right) = k^d-1$.
\end{thm}

The rest of the paper is structured as follows.  At the beginning of Section~\ref{satsection}, we review some fundamental properties of hypercubes which will be used in the proof of Theorem~\ref{main}. Then, we prove Theorem~\ref{main} by giving an explicit construction of a $(Q_d,Q_m)$-saturated graph with bounded average degree. In Section~\ref{wsatSection}, we turn our attention to weak saturation and prove Theorems~\ref{wsatThmQ},~\ref{wsatThm} and~\ref{cycleThm}. We conclude in Section~\ref{concl} by mentioning a number of open problems.

\section{Minimum Saturation in the Hypercube}
\label{satsection}

\subsection{Preliminaries}

Let $e_j$ denote the $j$th standard basis vector in $\mathbb{F}_2^d$; ie., the vector in which the $j$th coordinate is equal to one and every other coordinate is zero. Given a vertex $v$ of $Q_d$, let $|v|$ denote the sum of its coordinates. A basic fact about hypercubes is that, for $d\geq m$, every copy $Q$ of $Q_m$ in $Q_d$ is induced by a set of vertices of the form $\left\{v+\sum_{j\in J'}e_j:J'\subseteq J\right\}$ where $v$ is a fixed vertex and  $J\subseteq [d]$ has cardinality $m$. We say that the coordinates of $J$ and $[d]\setminus J$ are \emph{variable} and \emph{fixed} under $Q$, respectively.

We will need a standard result from Coding Theory, due to Hamming~\cite{Hamming} (for another reference, see~\cite{codebook,codebook2}). This result was also used by 
Johnson and Pinto~\cite{hypercube} (and we use some ideas from~\cite{hypercube} in our proof).

\begin{thm}[Hamming~\cite{Hamming}]
\label{hamThm}
There is an independent set $C\subseteq V\left(Q_{2^t-1}\right)$ such that every vertex of $V\left(Q_{2^t-1}\right)\setminus C$ has a unique neighbour in $C$. 
\end{thm}

A set $C$ as in Theorem~\ref{hamThm} is often referred to as a \emph{Hamming code}. We remark that, since every vertex of $Q_{2^t-1}$ has exactly $2^t-1$ neighbours, 
\begin{equation}\label{Cbound}|C| = \frac{2^{2^t-1}}{(2^t-1)+1} = 2^{2^t-t-1}.\end{equation}

Another ingredient of our proof of Theorem~\ref{main} is the following result of Conder~\cite{Conder}.

\begin{thm}[Conder~\cite{Conder}]
\label{3colours}
For every $s\geq1$, there is a $3$-colouring of the edges of $Q_s$ in which there is no monochromatic cycle of length $4$ or $6$.
\end{thm}

\subsection{Definitions and Proof Outline}

Our construction will use Theorems~\ref{hamThm} and~\ref{3colours} to build a spanning subgraph $G$ of $Q_d$ which contains no copy of $Q_m$ and which can be extended to a graph $G'$ which is $(Q_d,Q_m)$-saturated and satisfies the bound in Theorem~\ref{main}.

Throughout the proof, let $m\geq 2$ be fixed and let $d$ be an integer which we may choose to be sufficiently large. Let $t$ and $s$ be the unique integers such that
\begin{equation}\label{dval}d = 6m(2^t-1) + s\text{ and }0\leq s < 6m2^t.\end{equation}
We view the set $[d]$ as a union of $m$ intervals of length $6(2^t-1)$ indexed by the elements of $[m]$ followed by one interval of length $s$. Given $v\in V(Q_d)$ and $1\leq i\leq m$,  let $v(i)$ denote the vertex of $V(Q_{6(2^t-1)})$ obtained by restricting $v$ to the coordinates of the interval corresponding to $i$, and we let $v(m+1)$ be the vertex of $V(Q_s)$ obtained by restricting $v$ to its last $s$ coordinates. Each of the first $m$ intervals of $[d]$ is further divided into $6$ subintervals of length $(2^t-1)$ indexed by the pairs $(r,\gamma)$ where $r\in \{0,1\}$ and $\gamma\in\{0,1,2\}$. For $(i,r,\gamma)\in[m]\times\{0,1\}\times\{0,1,2\}$, let $v(i,r,\gamma)$ be the vertex of $V(Q_{2^t-1})$ obtained by restricting $v(i)$ to the subinterval corresponding to $(r,\gamma)$. 

Let $C$ be a subset of $V\left(Q_{2^t-1}\right)$ as in Theorem~\ref{hamThm}. We will treat the vertices  $v\in V\left(Q_d\right)$ differently depending on which of the triples $(i,r,\gamma)\in [m]\times \{0,1\}\times \{0,1,2\}$ satisfy $v(i,r,\gamma)\in C$. For starters, let $X$ denote the set of all vertices $v$ for which there exists some $i$ and $(r,\gamma)\neq (r',\gamma')$ such that $v(i,r,\gamma),v(i,r',\gamma')\in C$. The vertices of $X$ will be isolated in $G$ and will not play a large role in the construction. Now, divide the vertices of $V\left(Q_d\right)\setminus X$ into sets $A_0,\dots,A_m$ where $A_j$ is defined to be the set of all vertices of $V\left(Q_d\right)\setminus X$ for which there are exactly $j$ triples $(i,r,\gamma)$ such that $v(i,r,\gamma)\in C$. We remark that, by (\ref{Cbound}) and (\ref{dval}),
\begin{equation}\label{allButA0}\left|V\left(Q_d\right)\setminus A_0\right| = O\left(2^d/d\right)\end{equation}

Our goal is to construct a graph $G$ which does not contain $Q_m$ as a subgraph and does not contain any edge $uv$ of $Q_d$ where $u,v\in A_0$ such that, for any such edge, the graph $G+uv$ contains a copy of $Q_m$. Given such a graph $G$, let $G'$ be a graph obtained by adding a maximal set of edges to $G$ without creating a copy of $Q_m$. Then $G'$ is $\left(Q_d,Q_m\right)$-saturated and every edge of $G'$ has at least one endpoint in $V\left(Q_d\right)\setminus A_0$. By (\ref{allButA0}), the total number of edges incident to $V\left(Q_d\right)\setminus A_0$ is at most $O\left(2^d\right)$ and so we obtain $\sat(Q_d,Q_m)\leq \left|E\left(G'\right)\right| = O\left(2^d\right)$, as desired. 

The difficulty of the proof is to ensure that $G+e$ contains a copy of $Q_m$ for every edge $e$ of $Q_d$ joining two vertices of $A_0$ while simultaneously maintaining the property that $G$ does not contain a copy of $Q_m$. For the latter, we will apply Theorem~\ref{3colours} and a parity condition to ``break'' any potential copies of $Q_m$ in $G$. In what follows, let $\phi_1:E\left(Q_{6(2^t-1)}\right)\to\{0,1,2\}$ and $\phi_2:E\left(Q_s\right)\to\{0,1,2\}$ be colourings as in Theorem~\ref{3colours}, and let $\phi:E\left(Q_{6(2^t-1)}\right)\sqcup E\left(Q_s\right)\to\{0,1,2\}$ be the mapping which is equal to $\phi_1$ on $E\left(Q_{6(2^t-1)}\right)$ and equal to $\phi_2$ on $E\left(Q_s\right)$.

\subsection{The Construction}

We now give the details of our construction. 

\begin{step}
\label{dominate}
Add to $G$ every edge of $Q_d$ which joins a vertex of $A_j$ to a vertex of $A_{j+1}$ for $0\leq j\leq m-2$. 
\end{step}

\begin{step}
\label{parity}
Suppose that $uv$ is an edge of $Q_d$ such that $u,v\in A_j$ for some $1\leq j\leq m-1$. Let $k$ be the unique element of $[m+1]$ such that $u(k)\neq v(k)$. We add the edge $uv$ to $G$ if 
\begin{enumerate}[(i)]
\item either $k=m+1$ or there does not exist $(r',\gamma')\in \{0,1\}\times\{0,1,2\}$ such that $u(k,r',\gamma')=v(k,r',\gamma')\in C$,
\end{enumerate}
and for every $(i,r,\gamma)\in [m]\times \{0,1\}\times\{0,1,2\}$ such that $i\neq k$ and $v(i,r,\gamma)\in C$ we have
\begin{enumerate}[(i)]
\addtocounter{enumi}{1}
\item $\gamma = \phi(u(k)v(k))$,  and 
\item\label{mod2} $|v| + |v(k)| + |v(i)| +j-1 \equiv r\mod 2$.
\end{enumerate}
Note that (\ref{mod2}) is well defined since $u$ and $v$ differ only on a coordinate of the interval corresponding to $k$ and so for every $i\neq k$ we have
\[|v| + |v(k)| +|v(i)|\equiv |u|+|u(k)|+|u(i)| \mod 2.\]
\end{step}

Before moving on we make a few observations, each of which can be verified by looking carefully at Steps~\ref{dominate} and~\ref{parity}.

\begin{obs}
\label{isolates}
$A_m\cup X$ is a set of isolated vertices in $G$. 
\end{obs}

\begin{obs}
\label{A0stable}
$A_0$ is an independent set of $G$.
\end{obs}

\begin{obs}
\label{Cfixes}
If $uv$ is an edge of $G$ such that $u(i,r,\gamma)=v(i,r,\gamma)\in C$, then $u(i,r',\gamma') = v(i,r',\gamma')$ for every pair $(r',\gamma')\in\{0,1\}\times\{0,1,2\}$. 
\end{obs}

\begin{obs}
\label{m+1change}
If $uv$ is an edge of $G$ such that $u(m+1)\neq v(m+1)$, then $u,v\in A_j$ for some $1\leq j\leq m-1$.
\end{obs}

The next observation follows from  the fact that $C$ is an independent set in $Q_{2^t-1}$. 

\begin{obs}
\label{ctoc}
If $uv$ is an edge of $G$ such that $u(i,r,\gamma)\neq v(i,r,\gamma)$, then at most one of $u(i,r,\gamma)$ or $v(i,r,\gamma)$ is in $C$. 
\end{obs}

To complete the proof of Theorem~\ref{main}, it suffices to establish the following two claims. We state these claims now and show that they imply Theorem~\ref{main} before proving the claims themselves.

\begin{claim}
\label{saturated}
For every edge $uv$ of $Q_d$ such that $u,v\in A_0$, the graph $G+uv$ contains a copy of $Q_m$.
\end{claim}

\begin{claim}
\label{Qmfree}
$G$ does not contain $Q_m$ as a subgraph. 
\end{claim}

\begin{proof}[Proof of Theorem~\ref{main}]
If $G$ is not $(Q_d,Q_m)$-saturated, then by Claim~\ref{Qmfree} we can extend $G$ to a $(Q_d,Q_m)$-saturated graph $G'$ by adding a maximal set of edges which do not create a copy of $Q_m$. By Claim~\ref{saturated} none of these additional edges are between vertices in $A_0$ and so, by Observation~\ref{A0stable}, $A_0$ is an independent set in $G'$. Thus, every edge of $G'$ has at least one endpoint in $\left(\bigcup_{k=1}^mA_k\right)\cup X$. The total number of edges incident to vertices of this set is at most
\[d\left|\left(\bigcup_{k=1}^mA_k\right)\cup X\right| = d|A_1| + d\left|\left(\bigcup_{k=2}^mA_k\right)\cup X\right|\]
Note that $\left|\left(\bigcup_{k=2}^mA_k\right)\cup X\right| = O\left(|C|^22^{d-2(2^t-1)}\right)$, which is $O\left(2^d/d^2\right)$ by (\ref{Cbound}) and (\ref{dval}). Therefore the second term of the above expression is $o\left(2^d\right)$ and it suffices to bound $d|A_1|$. We have
\[d|A_1| = d\left(6m|C|2^{d-(2^t-1)}\right) = d\left(6m2^{2^t-t-1}2^{d-(2^t-1)}\right)=6md2^{d-t}\]
\[= 6m\left(6m(2^t-1) + s\right)2^{d-t} < 72m^22^d\]
by (\ref{dval}). The result follows.
\end{proof}

\begin{rem}
\label{speciald}
Note that if $d$ is of the form $6m(2^t-1)$ for some $t$, then we obtain a better bound: $\sat(Q_d,Q_m)\leq (1+o(1))36m^22^d$. 
\end{rem}

Thus, it suffices to prove Claims~\ref{saturated} and~\ref{Qmfree}.

\begin{proof}[Proof of Claim~\ref{saturated}]
Let $uv$ be an edge of $Q_d$ where $u,v\in A_0$ and let $k$ be the unique element of $[m+1]$ for which $u(k)\neq v(k)$. Define $\gamma := \phi(u(k)v(k))$ and let $I\subseteq [m]\setminus\{k\}$ be a set of size $m-1$. For $i\in I$, pick $r_i\in\{0,1\}$ so that
\[|v| + |v(k)| + |v(i)|\equiv r_i\mod 2.\]

For each $i\in I$, let $c_i$ be the unique neighbour of $v(i,r_i,\gamma)$ in $Q_{2^t-1}$ contained in $C$. Given $I'\subseteq I$, we let $v_{I'}$ and $u_{I'}$ be the vertices of $Q_d$ such that $u_{I'}(i,r_i,\gamma)= v_{I'}(i,r_i,\gamma)=c_i$ for all $i\in I'$ and, on all other coordinates, $u_{I'}$ and $v_{I'}$ agree with $u$ and $v$, respectively. In particular, $v_\emptyset=v$ and $u_\emptyset = u$. If $I'\neq I$, then for any $i'\in I\setminus I'$ we see that $v_{I'}$ is adjacent to $v_{I'\cup\{i'\}}$ and $u_{I'}$ is adjacent to $u_{I'\cup\{i'\}}$ in $G$ via edges added in Step~\ref{dominate}. Also, for $I'\subseteq I$, we have $v_{I'},u_{I'}\in A_{|I'|}$ and, if $I'\neq \emptyset$, then for each $i\in I'$,
\[\left|v_{I'}\right| + \left|v_{I'}(k)\right| + \left|v_{I'}(i)\right|\equiv r_i + |I'|-1\mod 2.\]
Thus,  for $I'\neq\emptyset$, $v_{I'}$ is adjacent to $u_{I'}$ via an edge of $G$ added in Step~\ref{parity}. This implies that $\left\{v_{I'}: I'\subseteq I\right\}\cup \left\{u_{I'}: I'\subseteq I\right\}$ induces a copy of $Q_m$ in $G+uv$, which completes the proof. 
\end{proof}

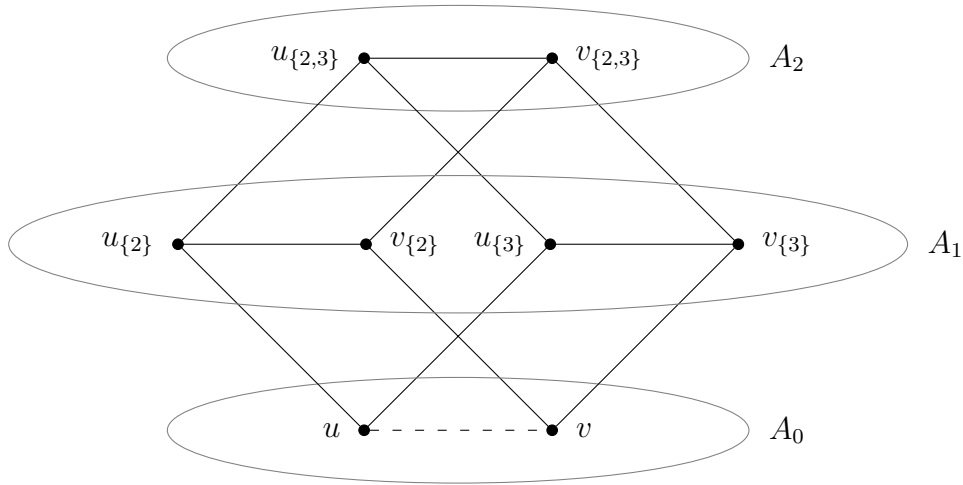
\begin{figure}[htbp]
\begin{center}
\begin{tikzpicture}[label distance=2.0mm]
    \tikzstyle{every node}=[draw,circle,fill=black,minimum size=4pt,
                            inner sep=0pt]
	

	\draw (0,0) node (u23) [label=left:$u_{\{2,3\}}$] {}
		    ++ (0:1.25cm) node[gray,ellipse,fill=none,minimum height=40pt, minimum width = 220pt, label=right:$A_2$] (mid){}
		    ++ (0:1.25cm) node (v23) [label=right:$v_{\{2,3\}}$] {}
		-- ++ (315:3.5cm) node (v3) [label=right:$v_{\{3\}}$] {}
		-- ++ (180:2.5cm) node (u3) [label=left:$u_{\{3\}}$] {}
		-- ++ (225:3.5cm) node (u) [label=left:$u$] {}
		    ++ (0:1.25cm) node[gray,ellipse,fill=none,minimum height=40pt, minimum width = 220pt, label=right:$A_0$] (mid){}
		 ++ (0:1.25cm) node (v) [label=right:$v$] {}
		-- ++ (135:3.5cm) node (v2) [label=right:$v_{\{2\}}$] {}
		-- ++ (180:2.5cm) node (u2) [label=left:$u_{\{2\}}$] {};
        
	\draw[ dash pattern=on 0.12cm off 0.18cm] (u)--(v);
	\draw (u23)--(v23);
	\draw (u)--(u2);
   	\draw (u2)--(u23);
	\draw (v)--(v3);
	\draw (u3)--(u23);
	\draw (v2)--(v23);



\path (0,0) coordinate (C1) -- ++(0:1.25cm) coordinate (C2) --  ++(315:3.5cm) coordinate(C4) -- ++ (135:3.5cm) coordinate(C5) -- ++(225:3.5) coordinate(C6) -- coordinate (C3) (C4);


\node[gray,ellipse, fill=none,minimum height=52pt, minimum width = 340pt, label=right:$A_1$] at (C3){};

\end{tikzpicture}
\end{center}
\caption{An illustration of the proof of Claim~\ref{saturated} in the case that $m=3$ and  $u(1)\neq v(1)$. }
\label{satFig}
\end{figure}

\begin{proof}[Proof of Claim~\ref{Qmfree}]
Suppose, to the contrary, that $G$ contains a copy $Q$ of $Q_m$. Let $J\subseteq [d]$ be the set of coordinates which are variable under $Q$. For each $i\in [m+1]$, let $J(i)$ be the set of coordinates of $J$ which are in the interval of $[d]$ corresponding to $i$. Moreover, given $(i,r,\gamma)\in [m]\times\{0,1\}\times\{0,1,2\}$, let $J(i,r,\gamma)$ be the set of coordinates of $J$ which are in the interval of $[d]$ corresponding to $(i,r,\gamma)$. 

\begin{subclaim}
\label{samespot}
For each $i\in [m]$ there is a pair $\left(r_i,\gamma_i\right)\in\{0,1\}\times\{0,1,2\}$ such that $J(i)=J\left(i,r_i,\gamma_i\right)$.
\end{subclaim}

\begin{proof}
Suppose to the contrary that there exists $(r,\gamma)\neq (r',\gamma')$ such that both $J(i,r,\gamma)$ and $J(i,r',\gamma')$ are non-empty. 

First, suppose that there exists a vertex $v\in V(Q)$ with $v(i,r,\gamma)\in C$. Let $u$ be a neighbour of $v$ in $Q$ obtained by changing a coordinate in $J(i,r',\gamma')$. By Observation~\ref{Cfixes}, the edge $uv$ is not contained in $G$, which is a contradiction. Thus, $v(i,r,\gamma),v(i,r',\gamma')\notin C$ for all $v\in V(Q)$. 

Now, let $Q'$ be a copy of $Q_2$ in $Q$ obtained by starting with an arbitrary vertex of $Q$ and varying one coordinate of $J(i,r,\gamma)$ and one coordinate of $J(i,r',\gamma')$. By the result of the previous paragraph, we see that $V(Q')\subseteq A_j$ for some $j$ and so every edge of $Q'$ was added in Step~\ref{parity}. By taking the edges of $Q'$ and restricting them to the interval of $[d]$ corresponding to $i$, we see that all of the resulting edges must receive the same colour under $\phi$. This contradicts the fact that there is no copy of $Q_2$ which is monochromatic under $\phi$ and completes the proof of the subclaim. 
\end{proof}

\begin{subclaim}
\label{vvvc}
Suppose that $|J(i)|\geq 2$ for some $i\in [m]$ and let $Q'$ be a copy of $Q_2$ in $Q$ obtained by starting with an arbitrary vertex and varying two coordinates of $J(i)$. Then there is a unique vertex $v\in V(Q')$ with $v(i,r_i,\gamma_i)\in C$. 
\end{subclaim}

\begin{proof}
By Subclaim~\ref{samespot} we have $J(i)=J(i,r_i,\gamma_i)$. By Observation~\ref{ctoc}, if $uv$ is an edge of $Q'$, then we cannot have $u(i,r_i,\gamma_i),v(i,r_i,\gamma_i)\in C$. Thus, if there are two vertices $u,v$ of $Q'$ for which $u(i,r_i,\gamma_i),v(i,r_i,\gamma_i)\in C$, then they must be non-adjacent. So, in this case, the vertices of $Q'$ alternate between $A_j$ and $A_{j+1}$ for some $j$ and every edge of $Q'$ was added in Step~\ref{dominate}. However, this implies that for $w\in V(Q')$ such that $w(i,r_i,\gamma_i)\notin C$, the vertex $w(i,r_i,\gamma_i)$ of $Q_{2^t-1}$ must have two distinct neighbours in $C$, which is a contradiction.

Now we assume that every vertex $v$ of $Q'$ satisfies $v(i,r_i,\gamma_i)\notin C$. However, in this case, we see that every edge of $Q'$ was added in Step~\ref{parity}. As before, we obtain a copy of $Q_2$ in $Q_{6(2^t-1)}$ which is monochromatic under $\phi$, a contradiction. 
\end{proof}

\begin{subclaim}
\label{atmost2}
$|J(i)|\leq 2$ for every $i\in[m]$. 
\end{subclaim}

\begin{proof}
Suppose not. Let $Q'$ be a copy of $Q_3$ in $Q$ obtained by starting with an arbitrary vertex and varying three coordinates of $J(i)$. If we vary any pair of these coordinates, leaving the third one fixed, we obtain a copy of $Q_2$ in $Q'$ which must obey Subclaim~\ref{vvvc}. This implies that there are precisely two vertices $x,y\in V(Q')$ such that $x(i,r_i,\gamma_i),y(i,r_i,\gamma_i)\in C$ and they are at distance $3$ in $Q'$. However, now we get that the vertices of $V(Q')\setminus \{x,y\}$ induce a copy $F$ of $C_6$ in $Q$, where every edge of $F$ was added in Step~\ref{parity}. Taking the edges of $F$ and restricting them to the interval of $[d]$ corresponding to $i$, we see that all such edges must receive the same colour under $\phi$. This contradicts the fact that there is no copy of $C_6$ which is monochromatic under $\phi$ and completes the proof of the subclaim. 
\end{proof}

\begin{subclaim}
\label{m+1<=1}
$|J(m+1)|\leq 1$.
\end{subclaim}

\begin{proof}
If not, then let $Q'$ be a copy of $Q_2$ in $Q$ obtained by starting at an arbitrary vertex and varying two coordinates of $J(m+1)$. Then, restricting the edges of $Q'$ to the last $s$ coordinates, we obtain a copy of $Q_2$ in $Q_s$ which is monochromatic under $\phi$, a contradiction.
\end{proof}

\begin{subclaim}
\label{vtov}
Suppose that $uv$ and $vw$ are distinct edges of $Q$ where $u,v,w\in A_j$ for some $j$. Then there is some $k\in[m]$ such that $u(k)\neq v(k)$ and $v(k)\neq w(k)$. 
\end{subclaim}

\begin{proof}
We show that there cannot exist distinct integers $k,k'\in [m+1]$ such that $u(k)\neq v(k)$ and $v(k')\neq w(k')$. This is sufficient to prove the subclaim since, by Subclaim~\ref{m+1<=1}, we cannot have $u(m+1)\neq v(m+1)$ and $v(m+1)\neq w(m+1)$.

Suppose that such integers $k,k'$ exist, and let $x$ be the unique vertex of $Q_d$ distinct from $v$ which is joined to both $u$ and $w$ in $Q_d$. Note that both of the edges $xu$ and $wx$ are present in $Q$. Also, $x\in A_j$ and so all of the edges $uv,vw,wx,xu$ were added to $G$ in Step~\ref{parity}. This implies that $j\geq1$ and that there is a triple $(i,r,\gamma)$ with $i\notin\{k,k'\}$ such that
\[u(i,r,\gamma)=v(i,r,\gamma)= w(i,r,\gamma)=x(i,r,\gamma)\in C.\]
However, one can easily check that the parity of $|v| + |v(k)| + |v(i)|$ is different from the parity of $|x| + |x(k)| + |x(i)|$ modulo $2$ and so only one can be equivalent to $r$. By definition of Step~\ref{parity}, only one of the edges $uv$ and $xw$ can exist in $G$. This contradiction completes the proof. 
\end{proof}

\begin{subclaim}
\label{J(m+1)=1}
If $|J(m+1)|=1$, then $|J(i)|\leq 1$ for all $i\in[m]$.
\end{subclaim}

\begin{proof}
If not, let $Q'$ be a copy of $Q_3$ in $Q$ obtained by starting at an arbitrary vertex and varying one coordinate of $J(m+1)$ and two of $J(i)$. Then, by Subclaim~\ref{vvvc}, there is an edge $uv\in E(Q')$ such that $u(i)\neq v(i)$ and $u,v\in A_j$ for some $j$ (in fact, there are many such edges). Now, let $w$ be the neighbour of $v$ in $Q'$ obtained by changing the coordinate of $J(m+1)$. By Observation~\ref{m+1change}, we have that $w\in A_j$ as well. However, this contradicts Subclaim~\ref{vtov}. 
\end{proof}

\begin{subclaim}
\label{J(i)=2}
There is at most one $i\in[m]$ for which $|J(i)|=2$. 
\end{subclaim}

\begin{proof}
If not, let $i,i'\in[m]$ such that $|J(i)|=|J(i')|=2$ and let $Q'$ be a copy of $Q_4$ obtained by starting at an arbitrary vertex and varying two coordinates of $J(i)$ and two coordinates of $J(i')$. By Subclaim~\ref{vvvc}, we see that there must exist edges $uv,vw\in E(Q')$ such that $u(i)\neq v(i)$, $v(i')\neq w(i')$ and $u,v,w\in A_j$ for some $j$. This contradicts Subclaim~\ref{vtov} and completes the proof. 
\end{proof}

Now, let us complete the proof of the claim. Throughout, for each $i\in[m]$, we let $(r_i,\gamma_i)$ be a pair such that $J(i)=J(i,r_i,\gamma_i)$, which exists by Subclaim~\ref{samespot}. In what follows, we let $j^*$ be the minimum integer $j$ such that $V(Q)\cap A_j\neq\emptyset$ and let $v^*\in V(Q)\cap A_{j^*}$. Note that for every $i\in [m]$ such that $J(i)\neq \emptyset$ we must have $v^*(i,r_i,\gamma_i)\notin C$. If not, then starting with $v^*$ and changing a coordinate of $J(i)$ yields a vertex of $V(Q)\cap A_{j^*-1}$ by Observations~\ref{Cfixes} and~\ref{ctoc}, contradicting our choice of $j^*$. By Subclaims~\ref{J(m+1)=1} and~\ref{J(i)=2}, there is at most one $i\in[m]$ such that $J(i)=\emptyset$. This implies that 
\begin{equation}\label{jstar}j^*\in\{0,1\},\end{equation} where $j^*$ must equal zero if no such $i$ exists.  We divide the proof into cases.

\begin{case}
\label{Ji1}
$|J(i)|=1$ for every $i\in[m]$.
\end{case}

In this case, we must have $j^*=0$ and so $v^*\in A_0$. However, if we start with $v$ and change the coordinate of $J(i)$ for any $i$, then we obtain a vertex $u$ with $u(i,r_i,\gamma_i)\in C$ by Observation~\ref{A0stable}. Thus, changing all $m$ of these coordinates yields a vertex $w$ such that $w(i,r_i,\gamma_i)\in C$ for all $i\in[m]$; that is, $w\in A_m$. This is a contradiction since, by Observation~\ref{isolates}, $w$ is an isolated vertex and therefore cannot belong to $Q$. 

\begin{case}
\label{Jm+1case}
$|J(m+1)|=1$.
\end{case}

Let $u$ be the neighbour of $v^*$ in $Q$ obtained by changing the coordinate in $J(m+1)$. We must have $u,v^*\in A_1$ by (\ref{jstar}) and Observation~\ref{m+1change}. This implies that there is some $i\in[m]$ such that $J(i)=\emptyset$ and $v^*(i,r_i,\gamma_i)\in C$. Also, by Subclaim~\ref{J(m+1)=1} and the Pigeonhole Principle, we must have $|J(i')|=1$ for every $i'\in[m]\setminus \{i\}$. 

Now, for each $i'\in[m]\setminus\{i\}$, let $w_{i'}$ be the neighbour of $v^*$ in $Q$ obtained by changing the coordinate in $J(i')$. By Subclaim~\ref{vtov} and the fact that $u,v^*\in A_1$, we cannot have $w_{i'}\in A_1$. So, by our choice of $j^*$, we must have $w_{i'}(i',r_{i'},\gamma_{i'})\in C$. As in the proof of Case~\ref{Ji1}, if we start with $v^*$ and change the coordinate of $J(i')$ for every $i'\in[m]\setminus \{i\}$, we obtain a vertex of $A_m$ contained in $Q$, contradicting Observation~\ref{isolates}. This completes the proof in this case.

\begin{case}
$|J(i)|=2$ for some $i\in[m]$.
\end{case}

By Subclaim~\ref{vvvc}, there is a neighbour $u$ of $v^*$ in $Q$ obtained by changing a coordinate in $J(i)$ such that $u\in A_{j^*}$. This immediately implies that $j^*=1$ by (\ref{jstar}) and Observation~\ref{A0stable}. 

Now, for each $i'\in[m]$ for which $|J(i')|=1$, let $w_{i'}$ be the neighbour of $v^*$ in $Q$ obtained by changing the coordinate in $J(i')$. By Subclaim~\ref{vtov} and the fact that $u,v^*\in A_1$, we must have $w_{i'}(i',r_{i'},\gamma_{i'})\in C$. Thus, if we let $x$ be the vertex obtained from $v^*$ by changing the coordinate of $J(i')$ for every such $i'$, then $x\in A_{m-1}$. 

Let $Q'$ be the copy of $Q_2$ in $Q$ obtained by starting at $x$ and varying the two coordinates of $J(i)$. Then, by Subclaim~\ref{vvvc}, there must be some vertex $y$ of $V(Q')$ such that $y(i,r_i,\gamma_i)\in C$. However, this implies that $y\in A_m$, contradicting Observation~\ref{isolates}. This completes the proof of Claim~\ref{Qmfree} and of Theorem~\ref{main}. 
\end{proof}

\section{Weak Saturation}
\label{wsatSection}

\subsection{Hypercubes and Grids}

In this section, we discuss weak saturation in cubes and more generally in grids.
We will prove Theorem \ref{wsatThm}, which immediately implies Theorem \ref{wsatThmQ}. 

Weak saturation is part of a more general theory, known as \emph{bootstrap percolation}. In the $(F,H)$-\emph{graph bootstrap process}, we start with an initial set $S_0$ of `infected' vertices in a graph $F$ and, at the $i$th step of the process, a vertex $v\in V(F)\setminus \left(\bigcup_{j=0}^{i-1}S_j\right)$ becomes infected and is added to $S_i$ if there is a copy of $H$ in the subgraph of $F$ induced by $\left(\bigcup_{j=0}^{i-1}S_j\right)\cup\{v\}$ containing $v$. A natural extremal problem is to determine the size of the smallest initial set $S_0$ such that every vertex of $F$ is eventually infected; such a set is said to be \emph{$(F,H)$-percolating}. It is easily seen that weak saturation corresponds to a bootstrap process on the edges of a graph, rather than its vertices 
(i.e.~a bootstrap process on the line graph). For an introduction to the literature on bootstrap percolation, see for instance~\cite{LinePerc,BootHypercube,GraphBoot} and the references therein.

Here we are interested in the `edge version' of the problem for cubes and grids.  The `vertex version' of the problem was solved by Balogh, Bollob\'{a}s, Morris and Riordan~\cite{LinAlg},
who determined the minimum size of a subset of $V\left(P_k^d\right)$ which is percolating with respect to (the vertex sets of) axis aligned copies of $P_r^m$, and also the 
minimum size of a $(K_k^d,K_r^m)$-percolating set, for all $k\geq r\geq 2$ and $d\geq m\geq1$  (here $K_k^d$ is the graph with vertex set $\{0, \ldots, k-1\}^d$ and two vertices are adjacent if they differ in exactly one coordinate).  Somewhat surprisingly, the two quantities are the same. 

We will use the following simple linear algebraic lemma from~\cite{LinAlg}.  Given a graph $F$ and a set $\mathcal{H}$ of subgraphs of $F$, let $\wsat\left(F,\mathcal{H}\right)$ be the minimum number of edges in a graph $G$ such that the edges of $E(F)\setminus E(G)$ can be added to $G$, one edge at a time, in such a way that each added edge increases the number of graphs of $\mathcal{H}$ contained in $G$; such a graph is said to be \emph{weakly $(F,\mathcal{H})$-saturated}. In this language, $\wsat^*\left(P_k^d, P_r^m\right) = \wsat\left(P_k^d,\mathcal{H}\right)$ where $\mathcal{H}$ contains all axis aligned copies of $P_r^m$ in $P_k^d$. 

\begin{lem}[Balogh, Bollob\'{a}s, Morris and Riordan~\cite{LinAlg}]
\label{LinAlgLemma}
Let $F$ be a graph, let $\mathcal{H}$ be a collection of subgraphs of $F$, and let $W$ be a vector space. Suppose that there exists a set $\left\{f_e: e\in E(F)\right\}\subseteq W$ such that for every $H\in \mathcal{H}$ there are non-zero scalars $\left\{c_e:e\in E(H)\right\}$ such that $\sum_{e\in E(H)}c_ef_e=0$. Then
\[\wsat(F,\mathcal{H})\geq \vdim\left(\vspan\{f_e:e\in E(F)\}\right).\]
\end{lem}

\begin{proof}
Let $G$ be a weakly $(F,\mathcal{H})$-saturated graph. Define $G_0:=G$ and label the edges of $E(F)\setminus E(G)$ by $e_1,\dots,e_k$ such that for $1\leq i\leq k$ there is a subgraph $H_i\in\mathcal{H}$ of $G_i:=G_{i-1}+e_i$ such that $H_i$ contains $e_i$. Thus, by hypothesis, $f_{e_i}$ can be written as a linear combination of the vectors in $\left\{f_e:e\in E(H_i)\setminus\left\{e_i\right\}\right\}$. This implies that
\[\vspan\left\{f_e: e\in E(G_0)\right\} = \vspan\left\{f_e: e\in E(G_1)\right\} =\dots =\vspan\left\{f_e: e\in E(G_k)\right\}.\]
Since $G=G_0$ and $F=G_k$, it must be the case that $|E(G)|\geq \vdim\left(\vspan\left\{f_e: e\in E(F)\right\}\right)$. Since $G$ was an arbitrary weakly $(F,\mathcal{H})$-saturated graph, this completes the proof. 
\end{proof}

To handle the vertex version of the question, Balogh, Bollob\'{a}s, Morris and Riordan gave a clever construction of a suitable vector space and then proved that it has the same dimension as a suitable percolating set. The result then follows by Lemma~\ref{LinAlgLemma}. Our proof of Theorem \ref{wsatThm} uses the same approach, but requires a different construction. For other examples in which the notions of linear dependence and independence are applied to solve problems in weak saturation, see
Alon~\cite{Alonwsat}, Kalai~\cite{Hypercon} and Pikhurko~\cite{PikPhD,PikhurkoExt}.

 Given a vertex $v$ of $P_k^d$, say that a coordinate of $v$ is \emph{large} if it has size at least $r-1$ and \emph{small} otherwise. We let $L(v)$ denote the number of large coordinates of $v$ and let $|v|$ denote the sum of the coordinates of $v$. Given $i\in[d]$, a \emph{line} $\mathcal{L}$ in direction $i$ is a path of length $k-1$ in $P_k^d$ in which any two vertices of $\mathcal{L}$ differ on the $i$th coordinate.

\begin{proof}[Proof of Theorem~\ref{wsatThm}]
Let $G$ be a spanning subgraph of $P_k^d$ in which, for each vertex $v$,
\begin{itemize}
\item we add every edge $uv$ of $P_k^d$ such that $|u|=|v|-1$ and $u$ and $v$ differ on a small coordinate of $v$, and
\item we add $\min\{L(v),m-1\}$ edges $uv$ of $P_k^d$ such that $|u|=|v|-1$ and $u$ and $v$ differ on a large coordinate of $v$.
\end{itemize}
Some elementary counting gives us
\[E(G) = \sum_{j=0}^d\sum_{i=0}^{d-j}(m-1+i)\binom{d}{j}\binom{d-j}{i}(k-r+1)^j(r-2)^i\]
\[-\sum_{j=0}^{m-2}\sum_{i=0}^{d-j}(m-1-j)\binom{d}{j}\binom{d-j}{i}(k-r+1)^j(r-2)^i.\]
So, to prove the upper bound, we need only show that $G$ is weakly $\left(P_k^d,P_r^m\right)$-saturated. In order of increasing $|v|$, we add all missing edges of $E(P_k^d)\setminus E(G)$ from $v$ to vertices $u$ with $|u|=|v|-1$ one edge at a time. By construction, for every added edge $uv$ with $|u|=|v|-1$, the coordinate on which $u$ and $v$ differ is large in $v$. For any vertex $x$, if $L(x)\leq m-1$, then all edges $xy$ with $|x|=|y|-1$ are already present in $G$. Also, if $v$ is a vertex with $L(v)\geq m$, then $G$ contains $m-1$ edges $wv$ with $|w|=|v|-1$ such that the coordinate on which $w$ and $v$ differ is large in $v$. Putting this together, an easy inductive argument shows that, if we add edges in this order, then each edge added from $v$ to a vertex $u$ with $|u|=|v|-1$ creates a new copy of $P_r^m$ in which $v$ is the `top' vertex.

For the lower bound, we apply Lemma~\ref{LinAlgLemma} where $\mathcal{H}$ consists of all axis aligned copies of $P_r^m$ in $P_k^d$. So, it suffices to show that there exists a vector space $W$ and a set $\left\{f_e: e\in E\left(P_k^d\right)\right\}\subseteq W$ which satisfy the hypotheses of Lemma~\ref{LinAlgLemma} such that 
\begin{equation}\label{dimBound}\vdim\left(\vspan\left\{f_e: e\in E\left(P_k^d\right)\right\}\right)\geq |E(G)|.\end{equation}
The space $W$ that we choose is the direct sum of $k^d$ copies of $\mathbb{R}^{m-1}$, one for each vertex of $P_k^d$, and $dk^{d-1}$ copies of $\mathbb{R}^{r-2}$, one for each line  in $P_k^d$. That is,
\[W:=\left(\bigoplus_{x}\mathbb{R}^{m-1}\right)\oplus\left(\bigoplus_{\mathcal{L}}\mathbb{R}^{r-2}\right)\]
where $x$ ranges over all vertices of $P_k^d$ and $\mathcal{L}$ ranges over all lines in $P_k^d$. Given a vector $w$ of $W$, a vertex $x$ of $P_k^d$ and a line $\mathcal{L}$ of $P_k^d$, let $\pi_x(w)$ denote the projection of $w$ onto the copy of $\mathbb{R}^{m-1}$ corresponding to $x$ and let $\pi_\mathcal{L}(w)$ denote the projection of $w$ onto the copy of $\mathbb{R}^{r-2}$ corresponding to $\mathcal{L}$. Note that $w$ is determined by its projections.

Let $Z := \{z_1,\dots,z_d\}$ be a collection of $d$ vectors of $\mathbb{R}^{m-1}$ in general position. Also, let $Y:=\{y_1,\dots,y_{k-2}\}$ be a set of $k-2$ vectors of $\mathbb{R}^{r-2}$ such that $y_1,\dots,y_{r-2}$ are linearly independent and, for $r-1\leq t\leq k-1$, 
\begin{equation}\label{ydef}y_t := -\sum_{j=t-r+2}^{t-1}y_j.\end{equation}
Thus any consecutive $r-2$ vectors $y_i, \ldots, y_{i+ r-2}$ are linearly independent and any consecutive $r-1$ vectors $y_j, \ldots, y_{j+r-1}$ sum to zero.
Suppose that $e=uv$ is an edge of $P_k^d$ such that $u$ and $v$ differ on coordinate $i\in[d]$ and let $t$ be the maximum of the $i$th coordinates of $u$ and $v$. Further, let $\mathcal{L}$ be the unique line of $P_k^d$ containing $e$. We define $f_e$ to be the vector of $W$ such that 
\begin{itemize}
\item $\pi_u(f_e)=\pi_v(f_e) = z_i$ and $\pi_x(f_e)=0$ for every $x\in V(Q_d)\setminus\{u,v\}$, and
\item $\pi_\mathcal{L} = y_t$ and $\pi_{\mathcal{L}'} =0$ for every line $\mathcal{L}'\neq\mathcal{L}$.
\end{itemize}

In order to apply Lemma~\ref{LinAlgLemma}, we need to show that for every axis aligned copy $P$ of $P_r^m$ in $P_k^d$ there are non-zero scalars $\left\{c_e:e\in E(P)\right\}$ such that $\sum_{e\in E(P)}c_ef_e=0$. Let $P$ be an axis aligned copy of $P_r^m$ in $P_k^d$ and let $I\subseteq[d]$ be the set of $m$ coordinates which vary under $P$. Let $\{c_i: i\in I\}$ be a set of scalars such that $\sum_{i\in I}c_iz_i=0$. Note that $c_i\neq0$ for all $i\in I$ since the vectors of $Z$ are in general position.  

For each vertex $v$ of $P$, let $M(v)\subseteq [d]$ be the set of indices $j$ such that both $v-e_j$ and $v+e_j$ are contained in $P$. Define the \emph{lines} of $P$ to be the paths of length $r-1$ in $P$ obtained by taking the intersection of a line of $P_k^d$ with $V(P)$. Note that $|M(v)|$ is precisely the number of lines of $P$ in which $v$ has degree two. For each line $\mathcal{L}$ of $P$, define $m(\mathcal{L}):=|M(v)|$ where $v$ is an endpoint of $\mathcal{L}$ (clearly the endpoints give the same value). For each $i\in I$, let $E_i$ be the set of all edges of $P$ for which $i$ is the variable coordinate. For each edge $e\in E_i$ contained in a line $\mathcal{L}$ of $P$, define $d_e:=2^{m(\mathcal{L})}c_i$. We claim that $\sum_{e\in E(P)}d_ef_e=0$. 

First, let $\mathcal{L}$ be a line of $P$ in direction $i$. Then, for some $t\geq r-1$, we have
\[\pi_{\mathcal{L}}\left(\sum_{e\in E(P)}d_ef_e\right) = \sum_{e\in E(P)}d_e\pi_\mathcal{L}(f_e)= \sum_{e\in E(\mathcal{L})}2^{m(\mathcal{L})}c_i\pi_\mathcal{L}(f_e) = 2^{m(\mathcal{L})} c_i\sum_{j=t-r+2}^{t}y_j.\]
However, this sum is equal to zero by (\ref{ydef}).

Now, fix a vertex $v$ of $P$ and for each $i\in I$ let $\mathcal{L}_i$ be the line of $P$ in direction $i$ containing $v$. We have
\[\pi_v\left(\sum_{e\in E(P)}d_e f_e\right) = \sum_{e\in E(P)}d_e\pi_v(f_e)=\sum_{i\in M(v)}\sum_{e\in E_i}2^{m(\mathcal{L}_i)}c_i\pi_v(f_e) + \sum_{i\in I\setminus M(v)}\sum_{e\in E_i}2^{m(\mathcal{L}_i)}c_i\pi_v(f_e)\]
which is equal to
\[\sum_{i\in M(v)}2^{m(\mathcal{L}_i)+1}c_i\pi_v(f_e) + \sum_{i\in I\setminus M(v)}2^{m(\mathcal{L}_i)}c_i\pi_v(f_e)\]
by definition of $M(v)$. We observe that, for any $i\in M(v)$, we have $m(\mathcal{L}_i)=|M(v)|-1$ and for any $i\in I\setminus M(v)$ we have $m(\mathcal{L}_i)=|M(v)|$. Therefore, the above sum is equal to $2^{|M(v)|}\sum_{i\in I}c_iz_i$ which is zero by our choice of $\{c_i:i\in I\}$. Combining this with the result of the previous paragraph, we get $\sum_{e\in E(P)}d_ef_e=0$, as desired.

To complete the proof, it suffices to prove (\ref{dimBound}). To do so, we let $G$ be a graph as in the proof of the upper bound and show that the vectors of $\left\{f_e: e\in E(G)\right\}$ are linearly independent. Let $\{c_e: e\in E(G)\}$ be any set of scalars, not all of which are zero, and let $F$ be the spanning subgraph of $G$ containing all edges $e\in E(G)$ such that $c_e\neq 0$. Let $v$ be a vertex of non-zero degree in $F$ such that $|v|$ is maximum. 

First, consider the case that $v$ has degree at most $m-1$ in $F$. Let $J$ denote the set of $d_F(v)$ coordinates such that, for each $j\in J$, the edge from $v$ to $v-e_j$ is present in $F$. Then $\sum_{e\in E(G)}c_e\pi_v(f_e)$ is a linear combination of the vectors in $\{z_j: j\in J\}$ in which not all of the coefficients are zero. Since $|J|\leq m-1$ and the vectors of $Z$ are in general position, this sum is non-zero and therefore $\sum_{e\in E(G)}c_ef_e\neq 0$. 

Now, suppose that $v$ has degree at least $m$ in $F$. By construction of $G$, this implies that there is a coordinate $j$ which is small for $v$ such that the edge $e$ from $v$ to $v-e_j$ is present in $F$. Let $\mathcal{L}$ be the unique line of $P_k^d$ containing $e$. Then, by maximality of $|v|$, every edge of $\mathcal{L}$ contained in $F$ joins two vertices for which $j$ is a small coordinate. This implies that $\sum_{e\in E(F)}c_e\pi_\mathcal{L}(f_e)$ is a linear combination of the vectors in $\{y_i: 1\leq i\leq r-1\}$ in which not all of the coefficients are zero. Thus, since the vectors $y_1,\dots,y_{r-1}$ are linearly independent, this sum is non-zero and we obtain $\sum_{e\in E(G)}c_ef_e\neq 0$, which completes the proof.
\end{proof}

\begin{rem}
The proof of Theorem~\ref{wsatThm} implies that, for most values of $k\geq r\geq 2$ and $d\geq m\geq1$, there are many tight examples. For an example in two dimensions, see Figure~\ref{gridsFig}. 
\end{rem}

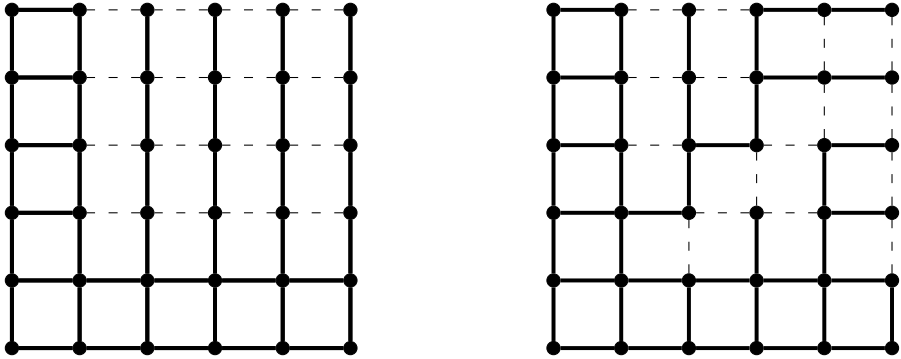
\begin{figure}[htbp]
\begin{center}
\begin{tikzpicture}[label distance=2.0mm]
    \tikzstyle{every node}=[draw,circle,fill=black,minimum size=5pt,
                            inner sep=0pt]


\draw (0,0) node (v00) {}
++ (0:0.9cm) node (v01) {}
++ (0:0.9cm) node (v02) {}
++ (0:0.9cm) node (v03) {}
++ (0:0.9cm) node (v04) {}
++ (0:0.9cm) node (v05) {}
++ (90:0.9cm) node (v15) {}
++ (180:0.9cm) node (v14) {}
++ (180:0.9cm) node (v13) {}
++ (180:0.9cm) node (v12) {}
++ (180:0.9cm) node (v11) {}
++ (180:0.9cm) node (v10) {}
++ (90:0.9cm) node (v20) {}
++ (0:0.9cm) node (v21) {}
++ (0:0.9cm) node (v22) {}
++ (0:0.9cm) node (v23) {}
++ (0:0.9cm) node (v24) {}
++ (0:0.9cm) node (v25) {}
++ (90:0.9cm) node (v35) {}
++ (180:0.9cm) node (v34) {}
++ (180:0.9cm) node (v33) {}
++ (180:0.9cm) node (v32) {}
++ (180:0.9cm) node (v31) {}
++ (180:0.9cm) node (v30) {}
++ (90:0.9cm) node (v40) {}
++ (0:0.9cm) node (v41) {}
++ (0:0.9cm) node (v42) {}
++ (0:0.9cm) node (v43) {}
++ (0:0.9cm) node (v44) {}
++ (0:0.9cm) node (v45) {}
++ (90:0.9cm) node (v55) {}
++ (180:0.9cm) node (v54) {}
++ (180:0.9cm) node (v53) {}
++ (180:0.9cm) node (v52) {}
++ (180:0.9cm) node (v51) {}
++ (180:0.9cm) node (v50) {};

\draw[line width = 1.5pt] (v01)--(v00);
\draw[line width = 1.5pt] (v02)--(v01);
\draw[line width = 1.5pt] (v03)--(v02);
\draw[line width = 1.5pt] (v04)--(v03);
\draw[line width = 1.5pt] (v05)--(v04);
\draw[line width = 1.5pt] (v00)--(v10);
\draw[line width = 1.5pt] (v01)--(v11);
\draw[line width = 1.5pt] (v11)--(v10);
\draw[line width = 1.5pt] (v02)--(v12);
\draw[line width = 1.5pt] (v12)--(v11);
\draw[line width = 1.5pt] (v03)--(v13);
\draw[line width = 1.5pt] (v13)--(v12);
\draw[line width = 1.5pt] (v04)--(v14);
\draw[line width = 1.5pt] (v14)--(v13);
\draw[line width = 1.5pt] (v05)--(v15);
\draw[line width = 1.5pt] (v15)--(v14);
\draw[line width = 1.5pt] (v10)--(v20);
\draw[line width = 1.5pt] (v11)--(v21);
\draw[line width = 1.5pt] (v21)--(v20);
\draw[line width = 1.5pt] (v12)--(v22);
\draw[ dash pattern=on 0.12cm off 0.18cm] (v22)--(v21);
\draw[line width = 1.5pt] (v13)--(v23);
\draw[ dash pattern=on 0.12cm off 0.18cm] (v23)--(v22);
\draw[line width = 1.5pt] (v14)--(v24);
\draw[ dash pattern=on 0.12cm off 0.18cm] (v24)--(v23);
\draw[line width = 1.5pt] (v15)--(v25);
\draw[ dash pattern=on 0.12cm off 0.18cm] (v25)--(v24);
\draw[line width = 1.5pt] (v20)--(v30);
\draw[line width = 1.5pt] (v21)--(v31);
\draw[line width = 1.5pt] (v31)--(v30);
\draw[line width = 1.5pt] (v22)--(v32);
\draw[ dash pattern=on 0.12cm off 0.18cm] (v32)--(v31);
\draw[line width = 1.5pt] (v23)--(v33);
\draw[ dash pattern=on 0.12cm off 0.18cm] (v33)--(v32);
\draw[line width = 1.5pt] (v24)--(v34);
\draw[ dash pattern=on 0.12cm off 0.18cm] (v34)--(v33);
\draw[line width = 1.5pt] (v25)--(v35);
\draw[ dash pattern=on 0.12cm off 0.18cm] (v35)--(v34);
\draw[line width = 1.5pt] (v30)--(v40);
\draw[line width = 1.5pt] (v31)--(v41);
\draw[line width = 1.5pt] (v41)--(v40);
\draw[line width = 1.5pt] (v32)--(v42);
\draw[ dash pattern=on 0.12cm off 0.18cm] (v42)--(v41);
\draw[line width = 1.5pt] (v33)--(v43);
\draw[ dash pattern=on 0.12cm off 0.18cm] (v43)--(v42);
\draw[line width = 1.5pt] (v34)--(v44);
\draw[ dash pattern=on 0.12cm off 0.18cm] (v44)--(v43);
\draw[line width = 1.5pt] (v35)--(v45);
\draw[ dash pattern=on 0.12cm off 0.18cm] (v45)--(v44);
\draw[line width = 1.5pt] (v40)--(v50);
\draw[line width = 1.5pt] (v41)--(v51);
\draw[line width = 1.5pt] (v51)--(v50);
\draw[line width = 1.5pt] (v42)--(v52);
\draw[ dash pattern=on 0.12cm off 0.18cm] (v52)--(v51);
\draw[line width = 1.5pt] (v43)--(v53);
\draw[ dash pattern=on 0.12cm off 0.18cm] (v53)--(v52);
\draw[line width = 1.5pt] (v44)--(v54);
\draw[ dash pattern=on 0.12cm off 0.18cm] (v54)--(v53);
\draw[line width = 1.5pt] (v45)--(v55);
\draw[ dash pattern=on 0.12cm off 0.18cm] (v55)--(v54);

\draw[line width = 1.5pt] (v01)--(v00);
\draw[line width = 1.5pt] (v02)--(v01);
\draw[line width = 1.5pt] (v03)--(v02);
\draw[line width = 1.5pt] (v04)--(v03);
\draw[line width = 1.5pt] (v05)--(v04);
\draw[line width = 1.5pt] (v00)--(v10);
\draw[line width = 1.5pt] (v01)--(v11);
\draw[line width = 1.5pt] (v11)--(v10);
\draw[line width = 1.5pt] (v02)--(v12);
\draw[line width = 1.5pt] (v12)--(v11);
\draw[line width = 1.5pt] (v03)--(v13);
\draw[line width = 1.5pt] (v13)--(v12);
\draw[line width = 1.5pt] (v04)--(v14);
\draw[line width = 1.5pt] (v14)--(v13);
\draw[line width = 1.5pt] (v05)--(v15);
\draw[line width = 1.5pt] (v15)--(v14);
\draw[line width = 1.5pt] (v10)--(v20);
\draw[line width = 1.5pt] (v11)--(v21);
\draw[line width = 1.5pt] (v21)--(v20);
\draw[line width = 1.5pt] (v12)--(v22);
\draw[line width = 1.5pt] (v13)--(v23);
\draw[line width = 1.5pt] (v14)--(v24);
\draw[line width = 1.5pt] (v15)--(v25);
\draw[line width = 1.5pt] (v20)--(v30);
\draw[line width = 1.5pt] (v21)--(v31);
\draw[line width = 1.5pt] (v31)--(v30);
\draw[line width = 1.5pt] (v22)--(v32);
\draw[line width = 1.5pt] (v23)--(v33);
\draw[line width = 1.5pt] (v24)--(v34);
\draw[line width = 1.5pt] (v25)--(v35);
\draw[line width = 1.5pt] (v30)--(v40);
\draw[line width = 1.5pt] (v31)--(v41);
\draw[line width = 1.5pt] (v41)--(v40);
\draw[line width = 1.5pt] (v32)--(v42);
\draw[line width = 1.5pt] (v33)--(v43);
\draw[line width = 1.5pt] (v34)--(v44);
\draw[line width = 1.5pt] (v35)--(v45);
\draw[line width = 1.5pt] (v40)--(v50);
\draw[line width = 1.5pt] (v41)--(v51);
\draw[line width = 1.5pt] (v51)--(v50);
\draw[line width = 1.5pt] (v42)--(v52);
\draw[line width = 1.5pt] (v43)--(v53);
\draw[line width = 1.5pt] (v44)--(v54);
\draw[line width = 1.5pt] (v45)--(v55);

\draw (7.2,0) node (u00) {}
++ (0:0.9cm) node (u01) {}
++ (0:0.9cm) node (u02) {}
++ (0:0.9cm) node (u03) {}
++ (0:0.9cm) node (u04) {}
++ (0:0.9cm) node (u05) {}
++ (90:0.9cm) node (u15) {}
++ (180:0.9cm) node (u14) {}
++ (180:0.9cm) node (u13) {}
++ (180:0.9cm) node (u12) {}
++ (180:0.9cm) node (u11) {}
++ (180:0.9cm) node (u10) {}
++ (90:0.9cm) node (u20) {}
++ (0:0.9cm) node (u21) {}
++ (0:0.9cm) node (u22) {}
++ (0:0.9cm) node (u23) {}
++ (0:0.9cm) node (u24) {}
++ (0:0.9cm) node (u25) {}
++ (90:0.9cm) node (u35) {}
++ (180:0.9cm) node (u34) {}
++ (180:0.9cm) node (u33) {}
++ (180:0.9cm) node (u32) {}
++ (180:0.9cm) node (u31) {}
++ (180:0.9cm) node (u30) {}
++ (90:0.9cm) node (u40) {}
++ (0:0.9cm) node (u41) {}
++ (0:0.9cm) node (u42) {}
++ (0:0.9cm) node (u43) {}
++ (0:0.9cm) node (u44) {}
++ (0:0.9cm) node (u45) {}
++ (90:0.9cm) node (u55) {}
++ (180:0.9cm) node (u54) {}
++ (180:0.9cm) node (u53) {}
++ (180:0.9cm) node (u52) {}
++ (180:0.9cm) node (u51) {}
++ (180:0.9cm) node (u50) {};

\draw[line width = 1.5pt] (u01)--(u00);
\draw[line width = 1.5pt] (u02)--(u01);
\draw[line width = 1.5pt] (u03)--(u02);
\draw[line width = 1.5pt] (u04)--(u03);
\draw[line width = 1.5pt] (u05)--(u04);
\draw[line width = 1.5pt] (u00)--(u10);
\draw[line width = 1.5pt] (u01)--(u11);
\draw[line width = 1.5pt] (u11)--(u10);
\draw[line width = 1.5pt] (u02)--(u12);
\draw[line width = 1.5pt] (u12)--(u11);
\draw[line width = 1.5pt] (u03)--(u13);
\draw[line width = 1.5pt] (u13)--(u12);
\draw[line width = 1.5pt] (u04)--(u14);
\draw[line width = 1.5pt] (u14)--(u13);
\draw[line width = 1.5pt] (u05)--(u15);
\draw[line width = 1.5pt] (u15)--(u14);
\draw[line width = 1.5pt] (u10)--(u20);
\draw[line width = 1.5pt] (u11)--(u21);
\draw[line width = 1.5pt] (u21)--(u20);
\draw[line width = 1.5pt] (u22)--(u21);
\draw[ dash pattern=on 0.12cm off 0.18cm] (u12)--(u22);
\draw[line width = 1.5pt] (u13)--(u23);
\draw[ dash pattern=on 0.12cm off 0.18cm] (u23)--(u22);
\draw[line width = 1.5pt] (u14)--(u24);
\draw[ dash pattern=on 0.12cm off 0.18cm] (u24)--(u23);
\draw[line width = 1.5pt] (u25)--(u24);
\draw[ dash pattern=on 0.12cm off 0.18cm] (u15)--(u25);
\draw[line width = 1.5pt] (u20)--(u30);
\draw[line width = 1.5pt] (u21)--(u31);
\draw[line width = 1.5pt] (u31)--(u30);
\draw[line width = 1.5pt] (u22)--(u32);
\draw[ dash pattern=on 0.12cm off 0.18cm] (u32)--(u31);
\draw[line width = 1.5pt] (u33)--(u32);
\draw[ dash pattern=on 0.12cm off 0.18cm] (u23)--(u33);
\draw[line width = 1.5pt] (u24)--(u34);
\draw[ dash pattern=on 0.12cm off 0.18cm] (u34)--(u33);
\draw[line width = 1.5pt] (u35)--(u34);
\draw[ dash pattern=on 0.12cm off 0.18cm] (u25)--(u35);
\draw[line width = 1.5pt] (u30)--(u40);
\draw[line width = 1.5pt] (u31)--(u41);
\draw[line width = 1.5pt] (u41)--(u40);
\draw[line width = 1.5pt] (u32)--(u42);
\draw[ dash pattern=on 0.12cm off 0.18cm] (u42)--(u41);
\draw[line width = 1.5pt] (u33)--(u43);
\draw[ dash pattern=on 0.12cm off 0.18cm] (u43)--(u42);
\draw[line width = 1.5pt] (u44)--(u43);
\draw[ dash pattern=on 0.12cm off 0.18cm] (u34)--(u44);
\draw[line width = 1.5pt] (u45)--(u44);
\draw[ dash pattern=on 0.12cm off 0.18cm] (u35)--(u45);
\draw[line width = 1.5pt] (u40)--(u50);
\draw[line width = 1.5pt] (u41)--(u51);
\draw[line width = 1.5pt] (u51)--(u50);
\draw[line width = 1.5pt] (u42)--(u52);
\draw[ dash pattern=on 0.12cm off 0.18cm] (u52)--(u51);
\draw[line width = 1.5pt] (u43)--(u53);
\draw[ dash pattern=on 0.12cm off 0.18cm] (u53)--(u52);
\draw[line width = 1.5pt] (u54)--(u53);
\draw[ dash pattern=on 0.12cm off 0.18cm] (u44)--(u54);
\draw[line width = 1.5pt] (u55)--(u54);
\draw[dash pattern=on 0.12cm off 0.18cm] (u45)--(u55);

\end{tikzpicture}
\end{center}
\caption{Two non-isomorphic weakly $\left(P_6^2,P_3^2\right)$-saturated graphs with $\wsat\left(P_6^2,P_3^2\right)$ edges.}
\label{gridsFig}
\end{figure}

\subsection{Cycles in the Grid}

We prove Theorem~\ref{cycleThm} using an elementary argument. 

\begin{proof}[Proof of Theorem~\ref{cycleThm}]
First, notice that $\wsat\left(P_k^d,C_{2\ell}\right)\geq k^d-1$ holds since any weakly $\left(P_k^d,C_{2\ell}\right)$-saturated graph must be connected. To prove the upper bound, we proceed by induction on $d+k$. 

Consider the base case $k=2$ and $d=\ell$. For $0\leq t\leq \ell$, let $L_t$ be the set of all vertices $v$ of $Q_\ell$ with $|v|=t$ and define $x_t$ to be the vertex of $L_t$  which consists of $t$ ones followed by  $\ell-t$ zeros. We let $G$ be the graph such that, for $1\leq t\leq \ell$, each vertex $v\in L_t$ is joined to a unique vertex of $L_{t-1}$ which is chosen in the following way:
\begin{itemize}
\item if $t=1$ or $v=x_t$, then $v$ is joined to $x_{t-1}$. 
\item if $t\geq 2$ and $v\neq x_t$,  then $v$ is joined to an arbitrary vertex $y\in L_{t-1}\setminus\{x_{t-1}\}$ such that $vy\in E(Q_\ell)$. 
\end{itemize}
It is clear that $G$ is a tree and so $|E(G)|=2^\ell-1$. For each vertex $v$, let $P_v$ be the unique path in $G$ from $v$ to $x_0$. Note that, by construction, if $v\notin \{x_1,\dots,x_\ell\}$ then $V(P_v)\cap \{x_1,\dots,x_\ell\}=\emptyset$. 

We show that $G$ is weakly $(Q_\ell,C_{2\ell})$-saturated. First, add each edge $x_\ell v$ where $v\in L_{\ell-1}\setminus\{x_{\ell-1}\}$. Each of these edges creates a copy of $C_{2\ell}$ by taking the paths $P_{x_\ell}$ and $P_{v}$ along with the edge $x_\ell v$. 

Next, for each $t=\ell-1,\dots,2$, in turn, we add every edge $uv$ where $u\in L_t\setminus\{x_t\}$ and $v\in L_{t-1}\setminus\{x_{t-1}\}$. When doing so we can assume, inductively, that for each $s$ such that $t<s\leq \ell$ every edge from $L_s\setminus\{x_s\}$ to $L_{s-1}\setminus\{x_{s-1}\}$ is present. In particular, this implies that there is a path $R_u$ of length $\ell-t$ from $u$ to $x_\ell$ which does not contain any vertex of $\{x_0,\dots,x_{\ell-1}\}$. Therefore, when adding the edge $uv$, we obtain a copy of $C_{2\ell}$ by taking the paths $P_v$, $P_{x_\ell}$, $R_u$ and the edge $uv$. 

After this, for each $t=\ell-1,\dots,2$, in turn, we add every edge $ux_{t-1}$ where $u\in L_t\setminus\{x_t\}$. In this case, let $w$ be a vertex of $L_t\setminus\{x_t,u\}$ such that there are paths $R_u$ and $R_w$ from $u$ to $x_\ell$ and $w$ to $x_\ell$, respectively, such that $V(R_u)\cap V(R_w)=\{x_\ell\}$. We obtain a copy of $C_{2\ell}$ by taking the paths $R_u$, $R_w$, $P_w$ and $P_{x_{t-1}}$ along with the edge $ux_{t-1}$.

Finally, for each $t=\ell-1,\dots,2$, in turn, we add every edge $x_tv$ where $v\in L_{t-1}\setminus\{x_{t-1}\}$. Let $j$ be an index on which $v$ is zero and let $P$ be a path from $x_\ell$ to $x_0$ which contains $e_j$ but is disjoint from $\{x_1,\dots,x_{\ell-1}\}$. Notice that $P$ does not contain any vertex of $P_v$. We obtain a copy of $C_{2\ell}$ by taking the paths $x_tx_{t+1}\dots x_\ell$, $P$, and $P_v$ along with the edge $x_tv$. This completes the proof of the base case. 

Now, we assume that $d+k>\ell+2$ and that the proposition holds for smaller values of $d+k$. We divide the proof into two cases. 

\begin{case2}
$d>\ell$.
\end{case2}

Let $G_{d-1}$ be a weakly $\left(P_k^{d-1},C_{2\ell}\right)$-saturated graph with $k^{d-1}-1$ edges. We construct a weakly $\left(P_k^d,C_{2\ell}\right)$-saturated graph $G$. For each vertex $v$ of $P_k^d$, let $v_{d-1}$ be the restriction of $v$ to its first $d-1$ coordinates and let $v_d$ be the last coordinate of $v$. In constructing $G$, we add every edge $uv$ of $P_k^d$ such that
\begin{itemize}
\item $u_{d-1}v_{d-1}\in E(G_{d-1})$ and $u_{d}=v_{d}=0$, or
\item $u_{d-1}=v_{d-1}$. 
\end{itemize}
It is clear that $|E(G)|=k^d-1$. For $0\leq t\leq k-1$, let $S_t$ be the set of all vertices $v$ such that $v_d=t$. Note that $S_t$ induces a copy of $P_k^{d-1}$ in $P_k^d$. 

Let us show that $G$ is weakly $\left(P_k^d,C_{2\ell}\right)$-saturated. First, add all edges $uv$ of $P_k^d$ where $u,v\in S_0$ in some order such that each added edge creates a new copy of $C_{2\ell}$. This is possible by definition of $G_{d-1}$. Now, for each $t=1,\dots,k-1$, in turn, we add all edges between vertices of $S_t$. By induction on $t$ we can assume that all edges between vertices in $S_{t-1}$ are already present. Given an edge $uv$ where $u,v\in S_t$, let $u',v'\in S_{t-1}$ where $u'_{d-1}=u_{d-1}$ and $v'_{d-1}=v_{d-1}$. We obtain a copy of $C_{2\ell}$ by taking the edges $u'u$, $uv$, $vv'$ and a path of length $2\ell-3$ from $v'$ to $u'$ in $S_{t-1}$. This completes the proof in this case. 

\begin{case2}
$k>2$.
\end{case2}

Let $G_{k-1}$ be a weakly $\left(P_{k-1}^{d},C_{2\ell}\right)$-saturated graph with $(k-1)^{d}-1$ edges. For each set $S\subseteq [d]$,  let $m(S)$ be the minimum element of $S$ and let $Y_S$ be the set of all vertices $v$ of $P_k^d$ such that $S$ is precisely the set of coordinates on which $v$ is equal to $k-1$. We add to $G$ every edge $uv$ of $P_k^d$ such that
\begin{itemize}
\item $u,v\in Y_\emptyset$ and $uv\in E(G_{k-1})$, or
\item $u\in Y_S$ and $v\in Y_{S\setminus\{m(S)\}}$ for some non-empty $S\subseteq [d]$. 
\end{itemize}
It is clear that $|E(G)|=k^d-1$. 

Let us show that $G$ is weakly $\left(P_k^d,C_{2\ell}\right)$-saturated. First, add all edges $uv$ where $u,v\in Y_\emptyset$ such that each added edge creates a new copy of $C_{2\ell}$. This is possible by definition of $G_{k-1}$. 

Next, for each non-empty set $S\subseteq [d]$, in order of increasing cardinality, add all edges $uv$ where $u,v\in Y_S$. Inductively, we can assume that all edges between vertices of $Y_{S\setminus\{m(S)\}}$ are already present. Note that $Y_{S\setminus\{m(S)\}}$ induces a copy of $P_{k-1}^{d-|S\setminus\{m(S)\}|}$ where $|S\setminus\{m(S)\}|\leq d-2$. Let $u'$ and $v'$ be the unique neighbours of $u$ and $v$ in $Y_{S\setminus\{m(S)\}}$, respectively. We obtain a copy of $C_{2\ell}$ by taking the edges $u'u$, $uv$, $vv'$ and a path of length $2\ell-3$ from $v'$ to $u'$ in $Y_{S\setminus\{m(S)\}}$.

Finally, for each set $S$, in order of increasing cardinality, and element $j$ of $S\setminus\{m(S)\}$, add every edge $uv$ where $u\in Y_S$ and $v$ differs from $u$ on coordinate $j$. Let $u'$ be the neighbour of $u$ in $Y_{S\setminus\{m(S)\}}$ and let $v'$ be the neighbour of $v$ in $Y_{S\setminus\{j,m(S)\}}$. If $\ell=2$, then we obtain a copy of $C_4$ by taking the edges $u'u$, $uv$, $vv'$ and $u'v'$. Otherwise, let $w$ be the vertex of $Y_{S\setminus\{m(S)\}}$ obtained from $u'$ by decreasing coordinate $m(S)$ by one. Further, let $w'$ be vertex of $Y_{S\setminus\{j,m(S)\}}$ obtained from $w$ by decreasing coordinate $j$ by one. We obtain a copy of $C_{2\ell}$ by taking the edges $w'w$, $wu'$, $u'u$, $uv$, $vv'$ and a path of length $2\ell-5$ from $v'$ to $w'$ in $Y_{S\setminus\{j,m(S)\}}$. Note that this is possible since $w'$ is adjacent to $v'$ and the subgraph induced by $Y_{S\setminus\{j,m(S)\}}$ is isomorphic to $P_{k-1}^{d-|S\setminus\{j,m(S)\}|}$. This completes the proof. 
\end{proof}

\section{Open Problems}\label{concl}

Many interesting open problems remain, and we mention some of them here.

\subsection{Minimum Saturation}

Recall that we were able to prove a better upper bound on $\sat(Q_d,Q_m)$  when $d$ is of the form $6m(2^t-1)$ than for general values of $d$ (see Remark~\ref{speciald}). It is not clear whether this is simply an artifact of our proof, or part of a more general phenomenon. We ask the following.

\begin{ques}
For fixed $m\geq2$, does $\lim_{d\to\infty}\frac{\sat(Q_d,Q_m)}{2^d}$ exist?
\end{ques}

As it stands, the best known lower bound on $\sat(Q_d,Q_m)$ is $(m-1-o(1))2^d$, which follows from Theorem~\ref{wsatThmQ} and (\ref{wsatlower}). We doubt that this bound is tight. We ask the following. 

\begin{ques}
For fixed $m\geq2$, let $c_m := \liminf_{d\to\infty} \sat(Q_d,Q_m)/2^d.$ Does $\frac{c_m}{m} \rightarrow \infty$?
\end{ques}

It would also be interesting to study saturation numbers in the grid.  

\begin{ques}
Fix $m\geq 2$, $r\geq 2$ and $k\geq\max\{r,3\}$, and let $d \rightarrow \infty$. Is it true that
\[\sat^*\left(P_k^d,P_r^m\right) = o\left(\left|E\left(P_k^d\right)\right|\right)?\]
How about
\[\sat^*\left(P_k^d,P_r^m\right) = O\left(k^d\right)?\]
\end{ques}

Here, the notation $\sat^*$ indicates that we are considering axis aligned copies of $P_r^m$ in $P_k^d$.

Moving beyond the finite case, one could also consider grids inside the infinite square lattice ${\mathbb Z}^d$, where vertices are adjacent if they
have $\ell_1$-distance 1.

\begin{ques}
For $d\ge m\ge 2$, what is the infimal density of a  $Q_m$-saturated subgraph of the $d$-dimensional square lattice?  
More generally, what about the $d$-dimensional lattice that is saturated with respect to axis aligned copies of $P_r^m$?
\end{ques}

Here the density is naturally defined as the $\limsup$ of densities inside large boxes.  For both of the last two questions, it would also be interesting to know what happens if we do not require our (sub)grids to be axis aligned. 

Recall that $K_k^d$ is the graph with vertex set $\{0,\dots,k-1\}^d$ in which two vertices are adjacent if they differ (by any amount) in exactly one coordinate. In one of the original papers on minimum saturated graphs, Erd\H{o}s, Hajnal and Moon~\cite{EHM} determined $\sat(K_k,K_r)$ for all values of $k$ and $r$. One could also consider the following multidimensional analogue.

\begin{prob}
\label{prodclique}
For $k\geq r\geq 2$ and $d\geq m\geq1$, determine $\sat\left(K_k^d,K_r^m\right)$.
\end{prob}

The problem of determining $\sat\left(K_k,C_{\ell}\right)$ was raised by Bollob\'{a}s~\cite{BollobookExtremal} and has been extensively studied. The case $\ell=3$ is trivial, and the cases $\ell=4$ and $\ell=5$ were solved by Ollmann~\cite{Ollmann} and Chen~\cite{C5,AllC5}, respectively. The first non-trivial bounds for general $\ell$ were given by Barefoot, Clark, Entringer, Porter, Sz{\'e}kely, and Tuza~\cite{Barefoot}, who proved that there are positive constants $c_1$ and $c_2$ such that, for $\ell\neq 8,10$,
\[(1+c_1/\ell)n\leq \sat\left(K_k,C_{\ell}\right)\leq (1+c_2/\ell)n.\]
The value of $c_2$ was improved by Gould,{ \L}uczak and Schmitt~\cite{Luczak}. Currently, the best general upper and lower bounds are due to F{\"u}redi and Kim~\cite{FurediCycleSat}. It is natural to ask about cycles in the hypercube.  

\begin{prob}
\label{cycleProb}
For $\ell\geq 2$ and $d\geq \log_2(2\ell)$, determine $\sat(Q_d,C_{2\ell})$.
\end{prob}
More generally, similar questions can be asked for cycles in $P_k^d$ or $K_k^d$.
We remark that the related problem of determining $\ex\left(Q_d,C_{2\ell}\right)$ was proposed by Erd\H{o}s~\cite{Erdos} and is very well studied; see, e.g.,~\cite{Baber,Hong,Brass,Chung,Conder,Conlon,Furedi2,Furedi}.

\subsection{Weak Saturation}

Recall that Theorem~\ref{wsatThm} provides the explicit value of the weak saturation number of axis aligned copies of $P_r^m$ in $P_k^d$. However, for most values of $m,d,r,k$, the graph that we construct contains several `bent' copies of $P_r^m$. We ask about the weak saturation number of general copies of $P_r^m$ in $P_k^d$. 

\begin{prob}
For general $1\leq m\leq d$ and $2\leq r\leq k$, determine $\wsat\left(P_k^d,P_r^m\right)$. 
\end{prob}

In~\cite{BollobookExtremal}, Bollob\'{a}s conjectured that $\wsat(K_k,K_r)=\sat(K_k,K_r)$ for all $k$ and $r$. This was proved for $r< 7$ in~\cite{wsat} and for general $r$ by Alon~\cite{Alonwsat} and Kalai~\cite{Kalai,Hypercon} using methods from exterior algebra and matroid theory. A multidimensional version of this problem also seems interesting. 

\begin{prob}
\label{wsatK}
For general $1\leq m\leq d$ and $2\leq r\leq k$, determine $\wsat\left(K_k^d,K_r^m\right)$.
\end{prob}

As was mentioned in Section~\ref{wsatSection}, the `vertex version' of Problem~\ref{wsatK} was solved in~\cite{LinAlg}. 

Theorem~\ref{cycleThm} leaves open the problem of determining $\wsat\left(P_k^d,C_{2\ell}\right)$ when $d<\ell$. We ask the following.

\begin{ques}
Given $k\geq 2$ and $\ell>d\geq2$ such that $P_k^d$ contains a cycle of length $2\ell$, what is the value of $\wsat\left(P_k^d,C_{2\ell}\right)$?
\end{ques}

\subsection{Semi-saturation}

Johnson and Pinto~\cite{hypercube} also study another type of saturation problem, which is often referred to as semi-saturation~\cite{mindeg,FurediCycleSat} (other terms have also been used: see \cite{BollobookExtremal,Kalai,kSperner,PikPhD, Tuza}). Given graphs $F$ and $H$, say that a spanning subgraph $G$ of $F$ is \emph{$(F,H)$-semi-saturated} if for every edge $e\in E(F)\setminus E(G)$, the graph $G+e$ contains more copies of $H$ than $G$ does. Note that $G$ may contain $H$ as a subgraph. Let $\ssat(F,H)$ denote the \emph{semi-saturation number}, which is the minimum number of edges in a $(F,H)$-semi-saturated graph. Then it is easy to see that
\[\wsat(F,H)\leq \ssat(F,H)\leq \sat(F,H).\]
Johnson and Pinto~\cite{hypercube} proved that $\ssat(Q_d,Q_m)\leq \left(m^2+m/2\right)2^d$. Based on this result and (\ref{Q2bound}), they asked whether $\ssat(Q_d,Q_2)$ is equal to $\sat(Q_d,Q_2)$ for sufficiently large $d$ and, if not, whether $\limsup_{d\to\infty} \frac{\sat(Q_d,Q_2)}{2^d} > \limsup_{d\to\infty}\frac{\ssat(Q_d,Q_2)}{2^d}$. Given Theorem~\ref{main}, it now seems natural to pose these questions more generally. 

\begin{ques}
For fixed $m\geq 2$, is $\ssat(Q_d,Q_m)$ is equal to $\sat(Q_d,Q_m)$ for sufficiently large $d$? If not, is $\limsup_{d\to \infty} \frac{\sat(Q_d,Q_m)}{2^d} > \limsup_{d\to\infty}\frac{\ssat(Q_d,Q_m)}{2^d}$?
\end{ques}

Of course, the open problems on saturation mentioned above also make sense for semi-saturation. 

\bibliography{QmSaturation}

\providecommand{\bysame}{\leavevmode\hbox to3em{\hrulefill}\thinspace}
\providecommand{\MR}{\relax\ifhmode\unskip\space\fi MR }
\providecommand{\MRhref}[2]{%
  \href{http://www.ams.org/mathscinet-getitem?mr=#1}{#2}
}
\providecommand{\href}[2]{#2}
\begin{thebibliography}{10}

\bibitem{Alonwsat}
N.~Alon, \emph{An extremal problem for sets with applications to graph theory},
  J. Combin. Theory Ser. A \textbf{40} (1985), no.~1, 82--89.

\bibitem{Baber}
R.~Baber, \emph{{Tur\'{a}n densities of hypercubes}}, arXiv:1201.3587v2,
  preprint, November 2012.

\bibitem{LinePerc}
P.~N. Balister, B.~Bollob\'{a}s, J.~D. Lee, and B.~P. Narayanan, \emph{{Line
  percolation}}, Random Structures \& Algorithms, to appear.

\bibitem{BootHypercube}
J.~Balogh and B.~Bollob{\'a}s, \emph{Bootstrap percolation on the hypercube},
  Probab. Theory Related Fields \textbf{134} (2006), no.~4, 624--648.

\bibitem{GraphBoot}
J.~Balogh, B.~Bollob{\'a}s, and R.~Morris, \emph{Graph bootstrap percolation},
  Random Structures Algorithms \textbf{41} (2012), no.~4, 413--440.

\bibitem{LinAlg}
J.~Balogh, B.~Bollob{\'a}s, R.~Morris, and O.~Riordan, \emph{Linear algebra and
  bootstrap percolation}, J. Combin. Theory Ser. A \textbf{119} (2012), no.~6,
  1328--1335.

\bibitem{Hong}
J.~Balogh, P.~Hu, B.~Lidick{\'y}, and H.~Liu, \emph{Upper bounds on the size of
  4- and 6-cycle-free subgraphs of the hypercube}, European J. Combin.
  \textbf{35} (2014), 75--85.

\bibitem{Barefoot}
C.~A. Barefoot, L.~H. Clark, R.~C. Entringer, T.~D. Porter, L.~A. Sz{\'e}kely,
  and Zs. Tuza, \emph{Cycle-saturated graphs of minimum size}, Discrete Math.
  \textbf{150} (1996), no.~1-3, 31--48, Selected papers in honour of Paul
  Erd{\H{o}}s on the occasion of his 80th birthday (Keszthely, 1993).

\bibitem{wsat}
B.~Bollob\'{a}s, \emph{Weakly {$k$}-saturated graphs}, Beitr\"age zur
  {G}raphentheorie ({K}olloquium, {M}anebach, 1967), Teubner, Leipzig, 1968,
  pp.~25--31.

\bibitem{BollobookExtremal}
\bysame, \emph{Extremal graph theory}, London Mathematical Society Monographs,
  vol.~11, Academic Press, London-New York, 1978.

\bibitem{Brass}
P.~Brass, H.~Harborth, and H.~Nienborg, \emph{On the maximum number of edges in
  a {$C_4$}-free subgraph of {$Q_n$}}, J. Graph Theory \textbf{19} (1995),
  no.~1, 17--23.

\bibitem{C5}
Y.-C. Chen, \emph{Minimum {$C_5$}-saturated graphs}, J. Graph Theory
  \textbf{61} (2009), no.~2, 111--126.

\bibitem{AllC5}
\bysame, \emph{All minimum {$C_5$}-saturated graphs}, J. Graph Theory
  \textbf{67} (2011), no.~1, 9--26.

\bibitem{Choi}
S.~Choi and P.~Guan, \emph{Minimum critical squarefree subgraph of a
  hypercube}, Proceedings of the {T}hirty-{N}inth {S}outheastern
  {I}nternational {C}onference on {C}ombinatorics, {G}raph {T}heory and
  {C}omputing, vol. 189, 2008, pp.~57--64.

\bibitem{Chung}
F.~R.~K. Chung, \emph{Subgraphs of a hypercube containing no small even
  cycles}, J. Graph Theory \textbf{16} (1992), no.~3, 273--286.

\bibitem{Conder}
M.~Conder, \emph{Hexagon-free subgraphs of hypercubes}, J. Graph Theory
  \textbf{17} (1993), no.~4, 477--479.

\bibitem{Conlon}
D.~Conlon, \emph{An extremal theorem in the hypercube}, Electron. J. Combin.
  \textbf{17} (2010), no.~1, Research Paper 111, 7 pages.

\bibitem{mindeg}
A.~N. Day, \emph{{Saturated graphs of prescribed minimum degree}},
  arXiv:1407.6664v1, preprint, July 2014.

\bibitem{Erdos}
P.~Erd{\H{o}}s, \emph{On some problems in graph theory, combinatorial analysis
  and combinatorial number theory}, Graph theory and combinatorics
  ({C}ambridge, 1983), Academic Press, London, 1984, pp.~1--17.

\bibitem{EHM}
P.~Erd{\H{o}}s, A.~Hajnal, and J.~W. Moon, \emph{A problem in graph theory},
  Amer. Math. Monthly \textbf{71} (1964), 1107--1110.

\bibitem{satsurvey}
J.~R. Faudree, R.~J. Faudree, and J.~R. Schmitt, \emph{{A survey of minimum
  saturated graphs}}, Electron. J. Combin. \textbf{18} (2011), Dynamic Survey
  19, 36 pp. (electronic).

\bibitem{FurediCycleSat}
Z.~F{\"u}redi and Y.~Kim, \emph{Cycle-saturated graphs with minimum number of
  edges}, J. Graph Theory \textbf{73} (2013), no.~2, 203--215.

\bibitem{Furedi2}
Z.~F{\"u}redi and L.~{\"O}zkahya, \emph{On 14-cycle-free subgraphs of the
  hypercube}, Combin. Probab. Comput. \textbf{18} (2009), no.~5, 725--729.

\bibitem{Furedi}
\bysame, \emph{On even-cycle-free subgraphs of the hypercube}, J. Combin.
  Theory Ser. A \textbf{118} (2011), no.~6, 1816--1819.

\bibitem{Luczak}
R.~Gould, T.~{\L}uczak, and J.~Schmitt, \emph{Constructive upper bounds for
  cycle-saturated graphs of minimum size}, Electron. J. Combin. \textbf{13}
  (2006), no.~1, Research Paper 29, 19 pages.

\bibitem{Hamming}
R.~W. Hamming, \emph{Error detecting and error correcting codes}, Bell System
  Tech. J. \textbf{29} (1950), 147--160.

\bibitem{hypercube}
J.~R. Johnson and T.~Pinto, \emph{{Saturated subgraphs of the hypercube}},
  arXiv:1406.1766v1, preprint, June 2014.

\bibitem{Kalai}
G.~Kalai, \emph{Weakly saturated graphs are rigid}, Convexity and graph theory
  ({J}erusalem, 1981), North-Holland Math. Stud., vol.~87, North-Holland,
  Amsterdam, 1984, pp.~189--190.

\bibitem{Hypercon}
\bysame, \emph{Hyperconnectivity of graphs}, Graphs Combin. \textbf{1} (1985),
  no.~1, 65--79.

\bibitem{codebook}
F.~J. MacWilliams and N.~J.~A. Sloane, \emph{The theory of error-correcting
  codes. {I}}, North-Holland Publishing Co., Amsterdam-New York-Oxford, 1977,
  North-Holland Mathematical Library, Vol. 16.

\bibitem{codebook2}
\bysame, \emph{The theory of error-correcting codes. {II}}, North-Holland
  Publishing Co., Amsterdam-New York-Oxford, 1977, North-Holland Mathematical
  Library, Vol. 16.

\bibitem{kSperner}
N.~Morrison, J.~A. Noel, and A.~Scott, \emph{On saturated {$k$}-{S}perner
  systems}, Electron. J. Combin. \textbf{21} (2014), no.~3, Paper 3.22, 17
  pages.

\bibitem{Ollmann}
L.~T. Ollmann, \emph{{$K_{2,2}$} saturated graphs with a minimal number of
  edges}, Proceedings of the {T}hird {S}outheastern {C}onference on
  {C}ombinatorics, {G}raph {T}heory, and {C}omputing ({F}lorida {A}tlantic
  {U}niv., {B}oca {R}aton, {F}la., 1972), Florida Atlantic Univ., Boca Raton,
  Fla., 1972, pp.~367--392.

\bibitem{PikPhD}
O.~Pikhurko, \emph{Extremal hypergraphs}, Ph.D. thesis, University of
  Cambridge, October 1999, pp.~i--xii and 1--150.

\bibitem{PikhurkoExt}
\bysame, \emph{Weakly saturated hypergraphs and exterior algebra}, Combin.
  Probab. Comput. \textbf{10} (2001), no.~5, 435--451.

\bibitem{Tuza}
Z.~Tuza, \emph{Asymptotic growth of sparse saturated structures is locally
  determined}, Discrete Math. \textbf{108} (1992), no.~1-3, 397--402,
  Topological, algebraical and combinatorial structures. Frol{\'{\i}}k's
  memorial volume.

\bibitem{Zykov}
A.~A. Zykov, \emph{On some properties of linear complexes}, Mat. Sbornik N.S.
  \textbf{24(66)} (1949), 163--188.

\end{thebibliography}
  \bibliographystyle{amsplain}
\end{document}